\documentclass[12pt]{amsart}
\usepackage{txfonts}
\usepackage{amscd,amssymb, amsmath}
\usepackage[colorlinks,plainpages,backref,urlcolor=blue]{hyperref}
\usepackage{verbatim}
\usepackage{xcolor}

\newtheorem{theorem}{Theorem}[section]
\newtheorem{corollary}[theorem]{Corollary}

\newtheorem{prop}[theorem]{Proposition}

\theoremstyle{definition}
\newtheorem{definition}[theorem]{Definition}
\newtheorem{example}[theorem]{Example}
\newtheorem{remark}[theorem]{Remark}

\newcommand{\C}{\mathbb{C}}

\newcommand{\PP}{\mathbb{P}}

\DeclareMathAlphabet{\pazocal}{OMS}{zplm}{m}{n}
\newcommand{\A}{{\pazocal{A}}}

\newcommand{\OO}{{\pazocal O}}

\newcommand{\cO}{\mathcal{O}}

\newcommand{\cE}{{\mathcal{E}}}

\newcommand{\CC}{{\mathcal{C}}}

\def\dot{\mathchar"013A}
\newcommand{\hdot}{{\raise1pt\hbox to0.35em{\Huge $\dot$}}}

\begin{document}
\date{January, 2025}

\title[Deletion-addition of a smooth conic for free curves]%
{Deletion-addition of a smooth conic for free curves}
\author{Anca~M\u acinic$^*$}
\thanks{$^*$ Partially supported by the project "Singularities and Applications'' - CF 132/31.07.2023 funded by the European Union - NextGenerationEU - through Romania's National Recovery and Resilience Plan.}



\begin{abstract}
We describe the behaviour of a free reduced plane projective curve with respect to the deletion, respectively addition, of a smooth conic. These results apply in particular to conic-line arrangements. We present some obstructions to the geometry and combinatorics of a free reduced curve, generalizing results known a priori only for free projective line arrangements. 
\end{abstract}

\maketitle

\section{Introduction}
\label{sec:introduction}

Let $\CC$ be a reduced curve in $\PP^2 \coloneqq \PP^2 \C$. If $\CC$  is the union of a finite number of lines and smooth conics, $\CC$ is called a {\it conic-line arrangement}, in short, a CL-arrangement.
One important motivation to study CL-arrangements is to explore generalizations in this direction of  the theory of hyperplane arrangements, as CL-arrangements are a natural generalization of arrangements of projective lines. 
{\it Addition-deletion} 
and {\it deletion-restriction}
 type results provide important inductive tools to approach various problems concerning arrangements of hyperplanes, for instance to study the freeness property. Among them, we mention the classical addition-deletion theorem of Terao (\cite{Terao}) and the more recent division theorem (\cite{A:Inv}) and  combinatorial deletion theorem  (\cite{A1})  of Abe. For curves, Schenck-Terao-Yoshinaga prove in  \cite{STY} an addition-type result,  that concerns the addition of a smooth curve, under some quasihomogeneity conditions.

An emblematic and long standing open problem in the theory of  arrangements of hyperplanes is Terao's Conjecture that predicts the combinatorial nature of the algebraic freeness property. Despite attracting considerable interest over the years, the conjecture is still open, even for arrangements of projective lines, i.e., in the $3$-dimensional case. This has generated an intense research on the topics of freeness and freeness-adjacent properties, such as the property of being plus-one generated, introduced in \cite{A:POG}. 

CL-arrangements too, and more generally, reduced curves in $\PP^2$,  were in recent years persistently studied in relation to the freeness property (see for instance \cite{ST, STY, AD, P}). Freeness for curves proves to be an even more complicated subject, compared to its manifestation in the field of hyperplane arrangements. As Schenck-Toh\u{a}neanu prove in \cite{ST}, freeness is not determined by the combinatorics of the curve (combinatorics of a curve as defined, for instance, in \cite{CM}). \\


In this note we investigate the effect, on the freeness property, of adding a smooth conic to a free reduced curve, respectively of deleting a smooth conic from a free reduced curve. Incidentally, these addition-deletion type results (Theorems  \ref{thm:del_conic_gen} and \ref{thm:add_conic_gen})  give an effective way to construct new free and plus-one generated examples of curves.

In particular we give an answer to an open question from \cite{MP1} about the type of CL-arrangement that can appear by deleting a smooth conic from a free  CL-arrangement. We formulate necessary and sufficient conditions for the resulting arrangement to be free or plus-one generated, in Theorem \ref{thm:del_conic}. 
Given a free CL-arrangement, under some quasihomogeneity assumptions, we show that one just has to count the points of intersection of the arrangement to the conic to decide if the result of the deletion of the conic from the arrangement is free, plus-one generated, or neither.  Moreover, we describe precisely the type of curve (from the point of view of a minimal resolution of its associated derivations module) we obtain in this third case, see Corollary \ref{cor:not_free_or_POG} and Remark  \ref{rem:del_case_3_detailed}.
  
As for the addition of a smooth conic to a CL-arrangement, the fact that it can produce a whole range of results, even in the quasihomogeneity hypothesis, was already known, and it is illustrated by \cite[Example 2.8]{MP1}. However, we give new insight on how the addition behaves, specifically we give the precise conditions for the result of the addition to be free,  plus-one generated or neither, see Theorem \ref{thm:add_conic}. We effectively describe a minimal resolution of the derivations module of the resulting curve in this third case, in Corollary \ref{cor:add_not_free_nPOG} and Remark \ref{rem:add_case_3_detailed}.

The main ingredients for the proofs of Theorems  \ref{thm:del_conic},  \ref{thm:add_conic} are found in Schenck-Toh\u{a}neanu's paper \cite{ST}. However, the authors are interested in   \cite{ST} only in the situation when the addition or deletion of a smooth conic applied to a free CL-arrangement results in a free CL-arrangement (see \cite[Theorem 3.4]{ST}), under a quasihomogeneity assumption. We take this a step further and we characterize completely the possible results to the addition or deletion of a smooth conic, starting from a free CL-arrangement. A more recent result from \cite{DS0}, Theorem \ref{thm: min_gens_sum}, also comes into play.

Moreover,  we show that the results concerning CL-arrangements, Theorems  \ref{thm:del_conic},  \ref{thm:add_conic} and Corollaries \ref{cor:comb_del}, \ref{cor:comb_add}  hold in a more general hypothesis, i.e.  for reduced curves, with or without the quasihomogeneity assumption. More  precisely, they hold for triples of reduced curves  with respect to a smooth conic (see Section \ref{sec:main} for the definition of a triple with respect to a smooth conic). For quasihomogeneous triples of reduced curves the statements are identical to the ones for quasihomogeneous triples of CL-arrangements. Dropping the quasihomogeneity assumption leads to some adjustments in the initial statements, see Theorems \ref{thm:del_conic_gen}, \ref{thm:add_conic_gen} and Theorem \ref{thm:cors_gen}.\\

To put the problem in context, we recall that recent results from \cite{Dimca} (see also \cite{MP2})  show that, when 
applied to a free CL-arrangement, any of the  two operations of deletion, respectively addition,  of a {\it line} from (respectively to) a CL-arrangement produces either a free or a  plus-one generated CL-arrangement.

 These types of results, but stated for arrangements of hyperplanes, appeared first in the work of  Abe (\cite{A:POG}), where the notion of {\it plus-one generated} was introduced. As seen in \cite{A:POG}, addition and deletion of a hyperplane behave perfectly symmetrical in a $3$-dimensional vector space:  both procedures give rise to a free or a plus-one generated arrangement. But in ambient spaces of higher dimension only the deletion works the same way, in full generality. The addition of a hyperplane to a free arrangement in spaces of dimension higher than $3$ can produce an arrangement which is neither free nor plus-one generated (see \cite[Example 7.5]{A:POG}).  In \cite{ADen}, Abe-Denham introduce,  for logarithmic forms,  an analogue notion to plus-one generated, {\it dual plus-one generated}, with respect to which the addition of a hyperplane to a hyperplane arrangement,  in arbitrary dimension, behaves nicely. That is, it produces as a result either a free or a dual plus-one generated arrangement. The two notions, plus-one generated and  dual plus-one generated,  are equivalent in a  $3$-dimensional vector space.\\
 
We present in Theorems \ref{thm:curve_line_comb} and  \ref{thm:cors_gen}
some geometric and combinatoric implications of the freeness property for reduced curves.
In particular, for a CL-arrangement, we obtain obstructions on its restrictions to lines and smooth conics and also obstructions on its intersection to lines and smooth conics that are not elements in the arrangement.

Notice that Theorem \ref{thm:curve_line_comb} is a generalization to free reduced curves of a result of Abe for free projective  line arrangements  from \cite{A2}  (recalled in Theorem \ref{thm:Abe_restr_free}), concerning the number of points of the restriction to a line of the arrangement, respectively the number of  points of the intersection of the arrangement to a line not contained in the arrangement.\\

 In Theorem \ref{thm:gen_restr_POG} we formulate similar geometric obstructions induced by the  plus-one generated property. With this we generalize a result on projective  line arrangements  from \cite{AIM}, recalled in Theorem \ref{thm:AIM_restr_POG}.\\
  
We conclude in Section \ref{sec:exp} with a series of examples that reflect the variety of situations from Theorems  \ref{thm:del_conic},  \ref{thm:add_conic} and their generalized versions, Theorems  \ref{thm:del_conic_gen} and \ref{thm:add_conic_gen}.

\section*{Acknowledgments}
I would like to thank Takuro Abe for clarifications on the notion of {\it dual plus-one generated} and duality in the realm of rank $2$ vector bundles. I would also like to thank Piotr Pokora for suggestions on a previous version of this preprint that led to a clearer exposition.

\section{Preliminaries}
\label{sec:prelim}

Let us briefly recall  the notion of  freeness for curves.
Denote $S:= \C[x,y,z]$ and let 
$$Der(S) \coloneqq S\partial_x \oplus S\partial_y  \oplus S\partial_z$$
 be the module of $S$-derivations.
To a reduced curve $\CC \subset \PP^2$, defined as the zero set of the  homogeneous polynomial $f_{\CC} \in S$,  one associates the module of logarithmic derivations
$$
D(\CC) \coloneqq \{\theta \in Der(S) \; |\; \theta(f_{\CC}) \in S f_{\CC}\}.
$$
$\CC$ is called {\it free} if $D(\CC)$ is free as an $S$-module. The degrees of a minimal set of homogeneous generators of $D(\CC)$ (which are independent of the choice of the generators themselves) constitute the {\it exponents} of the free curve $\CC$. 
The Euler derivation $\theta_E = x \partial_x + y \partial_y + z \partial_z$ is always an element of $D(\CC)$, so $1$ is always among the exponents.

Moreover, there is a decomposition of  the module of logarithmic derivations  as a direct sum of modules 
$$
D(\CC) = D_0(\CC) \oplus S \theta_E
$$
where 
$$ 
D_0(\CC) \coloneqq \{\theta \in Der(S) \; |\; \theta(f_{\CC}) = 0\}.
$$

\noindent So $\CC$ is free if and only if the $S$-module $D_0(\CC)$ is free. If $\CC$ is free with exponents $(1,a,b)$ then $D_0(\CC)$ is freely generated by two homogeneous derivations of degrees $a,b$. We will omit the exponent of $\CC$ corresponding to the Euler derivation and simply state that $\CC$ is free with exponents $(a,b)$.

In \cite[Definition 1.1]{A:POG} Abe introduces, for hyperplane arrangements in general, a notion close to and strongly related to freeness. A {\it plus-one generated} arrangement of hyperplanes is characterized by the existence of  a very simple resolution for the associated module of logarithmic derivations, a natural step away from the resolution of a free module. Later on,  in \cite{DS0}, Dimca-Sticlaru extend this definition to reduced plane projective curves, via their associated derivations modules:

\begin{definition}
\label{def:POG}
A reduced curve $\CC$ in $\PP^2$ is called plus-one generated with exponents $(a,b)$ and level $\ell$ if $D_0(\CC)$ admits a minimal resolution of the form:
$$0 \rightarrow  S(-\ell-1) \rightarrow  S(-\ell) \oplus S(-b) \oplus S(-a) \rightarrow D_0(\CC) \rightarrow 0$$
\end{definition}

This definition of the plus-one generated property is at first glance slightly different  from the definition given by the authors in \cite{DS0}. We clarify in Proposition \ref{prop:equiv_POG_def} that the two definitions actually coincide.\\

Let us recall some general facts about the derivations module $D_0(\CC)$ associated to a reduced degree $\mathit{d}$ curve $\CC \subset \PP^2$. It is a rank $2$, graded, reflexive, $S$-module, of projective dimension $pd_S(D_0(\CC)) \leq 2$, so  a minimal resolution is of type:

\begin{equation}
\label{eq: res_gen_deriv}
 0 \rightarrow  \bigoplus_{i=1}^{r-2} S[-e_i] \rightarrow  \bigoplus_{i=1}^{r} S[-d_i] \rightarrow D_0(\CC) \rightarrow 0
\end{equation}

Following  the convention established in \cite{DS0} we call the ordered multiset 
\begin{equation}
\label{eq: gen_exp}
(d_1 \leq \dots \leq d_r)
\end{equation}
{\it the (generalized) exponents} of $\CC$, where $d_1 \geq 1$. These are the ordered degrees of the  generators from a minimal set of homogeneous generators of $D_0(\CC)$, and do not depend on the choice of the generators.\\
 
 With this notion of generalized exponents defined for any reduced curve, let us prove that a plus-one generated curve $\CC$ with exponents $(a,b)$ and level $\ell$ as in  Definition \ref{def:POG} means precisely a curve with generalized exponents $(a,b,\ell)$ such that $a+b=\deg(\CC)$. This latter characterisation of the plus-one generated property was employed as definition in \cite{DS0}.
 
\begin{prop}
\label{prop:equiv_POG_def}
A curve $\CC$ is  plus-one generated with exponents $(a,b)$ and level $\ell$ if and only if $\CC$ has generalized exponents $(a,b,\ell)$ and $a+b=\deg(\CC)$.
\end{prop}

\begin{proof}
 Let $\CC$ be plus-one generated. Denote by  $\cE_{\CC}$ the sheafification of the $S$-module of logarithmic derivations $D_0(\CC)$. It is known that $\cE_{\CC}$ is in fact a locally free sheaf, i.e. a vector bundle.
Consider the exact sequence of vector bundles induced by the sheafification of the exact sequence of graded modules from Definition \ref{def:POG}:
\begin{equation}
\label{eq: sheaves_sequence_def_POG}
 0 \rightarrow  \cO_{\PP^2}(-1-\ell) \longrightarrow \cO_{\PP^2}(-\ell)\oplus \cO_{\PP^2}(-b)\oplus \cO_{\PP^2}(-a) \rightarrow \cE_{\CC}  \rightarrow 0
 \end{equation}
There is a well known formula that connects the Chern polynomials of the terms of an exact sequence of vector bundles, see for instance  \cite[Section 3]{F}. Using this formula we compute the first Chern number of $\cE_{\CC}$,   
 $$c_1(\cE_{\CC}) = 1-a-b.$$
Since, for reduced curves in general, $c_1(\cE_{\CC}) = 1 - \deg(\CC)$, it follows that the plus-one generated property implies
$$a+b=\deg(\CC).$$
Conversely, let us assume that $\CC$ is  a curve with generalized exponents $(a,b,\ell)$ such that $a+b=\deg(\CC)$. A minimal free resolution for $D_0(\CC)$ is of the following type:

$$0 \rightarrow  S(-e) \rightarrow  S(-\ell) \oplus S(-b) \oplus S(-a) \rightarrow D_0(\CC) \rightarrow 0$$

Just as before, pass to sheaves in the above exact sequence to compute $c_1(\cE_{\CC})$ and then use the equality $a+b=\deg(\CC)$ to conclude $$e = \ell+1.$$
\end{proof}

A very useful criterion to decide when a curve $\CC$ is free or plus-one generated involves the generalized exponents from \ref{eq: gen_exp}, see \cite[Theorem 2.3]{DS0}:

\begin{theorem}
\label{thm: min_gens_sum}
Let  $\CC \subset \PP^2$ be a reduced curve of degree $\mathit{d}$. Then, with the notations from \eqref{eq: res_gen_deriv}:
\begin{enumerate}
\item $\CC$ is free if and only if $d_1 + d_2 = d-1$.
\item $\CC$ is plus-one generated if and only if $d_1 + d_2 = d$.
\item $\CC$ is neither free nor plus-one generated if and only if $d_1 + d_2 > d$.
\end{enumerate}
\end{theorem}

For arrangements of lines in $\PP^2$, freeness imposes constraints on the number of points of the restriction to an arbitrary line in the arrangement, and also on the number of  intersection points of the arrangement to a an arbitrary line not contained in the arrangement.

\begin{theorem}(\cite{A2})
\label{thm:Abe_restr_free}
Let $\A$ be a free arrangement in $\PP^2$ with exponents $(a,b)$, $a \leq b$.
\begin{enumerate}
\item Let  $H \in \A$ and $\A' \coloneqq \A \setminus \{H\} $. Then either $|\A' \cap H| \leq a+1$ or $|\A' \cap H| = b+1$. 
\item Let  $L \subset \PP^2$ be a line such that  $L \notin \A$. Then $|\A \cap L| = a+1$ or $|\A \cap L| \geq  b+1$.
\end{enumerate}
\end{theorem}

We generalize the above result, proving that similar obstructions hold for the geometry of free reduced curves in $\PP^2$,  in Theorem \ref{thm:curve_line_comb}.\\

Likewise, it is known that the plus-one generated property imposes constraints on the cardinal of the restriction to an arbitrary line in the arrangement.

\begin{theorem}(\cite{AIM})
\label{thm:AIM_restr_POG}
Let $\A$ be a plus-one generated arrangement in $\PP^2$ with exponents $(a,b), \; a \leq b,$ and level $l$,  $H \in \A$ arbitrary and $\A' \coloneqq \A \setminus \{H\} $. Then either $|\A' \cap H| \leq a+1$ or $|\A' \cap H| \in \{b,b+1, l+1\}$.
\end{theorem}
 
We generalize Theorem \ref{thm:AIM_restr_POG}, showing that  the plus-one generated property for curves induces the same type of geometric obstructions as for projective line arrangements, in Theorem \ref{thm:gen_restr_POG}.\\

To state and prove the above mentioned results, Theorems \ref{thm:curve_line_comb} and  \ref{thm:gen_restr_POG},  we need to recall, in our particular setting,  the definition of an invariant associated to an  isolated hypersurface singularity,  introduced in \cite{Dimca}.\\

  Let  $\CC$ be an arbitrary reduced curve in $\PP^2$ and $p \in Sing(\CC)$. Denote by $\mu (\CC, p),\; \tau (\CC, p)$, the Milnor number, respectively, the Tjurina number, of the singularity $p$.  Let 
 $$
 \epsilon(\CC, p) \coloneqq \mu (\CC, p) - \tau (\CC, p)
 $$
be the measure of the defect from quasihomogeneity of $p \in Sing(\CC)$
(it is known that $p$ is quasihomogeneous if and only if the Milnor and Tjurina numbers are equal, i.e. $p$ quasihomogeneous is equivalent to $ \epsilon(\CC, p)=0$, see \cite{Saito}). In general, the inequality 
$\mu (\CC, p) \geq \tau (\CC, p)$ holds, i.e.
$$ \epsilon(\CC, p) \geq 0.$$

If  $\CC_1, \; \CC_2$ are two reduced curves  with no common irreducible components and  $p \in \CC_1 \cap  \CC_2$, denote 
$$
 \epsilon(\CC_1, \CC_2)_p \coloneqq  \epsilon(\CC_1 \cup  \CC_2,p) -  \epsilon(\CC_1,p)
$$
and 
$$
 \epsilon(\CC_1, \CC_2) \coloneqq \sum_{p \in \CC_1 \cap  \CC_2}  \epsilon(\CC_1, \CC_2)_p.
$$

\begin{theorem}
\label{thm:curve_line_comb}
Let $\CC$ be a reduced free curve in $\PP^2$ with exponents $(a,b), a \leq b$. 
\begin{enumerate}
\item  Let $L$ in $\PP^2$ be a line which is an irreducible component of $\CC$, i.e. $\CC = \CC' \cup L$, where $\CC'$ is the union of the irreducible components of $\CC$, other than $L$. Then either $|\CC' \cap L| \leq  a+1 - \epsilon(\CC',L)$ or  $|\CC' \cap L| = b+1 - \epsilon(\CC',L)$.
\item Let $L$ in $\PP^2$ be a  line which is not an irreducible component of $\CC$. Then either $|\CC \cap L| = a+1- \epsilon(\CC,L)$ or  $|\CC \cap L| \geq  b+1- \epsilon(\CC,L)$.
\end{enumerate}
\end{theorem} 

\begin{proof}
Immediately from \cite[Theorem 1.4]{Dimca} and \cite[Theorem 3.2]{MP2}.
\end{proof}
 
\begin{remark}
We formulate in Theorem \ref{thm:cors_gen}  geometric obstructions of the same flavour for free reduced curves, involving the restriction of the curve to one of its smooth conics (as irreducible component), respectively the intersection of the curve to an arbitrary smooth conic which is not an irreducible component of the curve.
\end{remark}

\begin{theorem}
\label{thm:gen_restr_POG}
Let $\CC$ be a plus-one generated curve with exponents $(a,b), \; a \leq b,$ and level $l$, and $L \subset \PP^2$ a line which is an irreducible component of $\CC$. Let $\CC'$ be the union of the irreducible components of $\CC$, other than $L$. Then either $|\CC' \cap L| + \epsilon(\CC', L)  \leq a+1$ or $|\CC' \cap L| + \epsilon(\CC', L)  \in \{b,b+1, l+1\}$.
\end{theorem}

\begin{proof}
Recall that, by \cite[Theorem 2.3]{Dimca} (extending and based on \cite[Theorem 1.6]{STY}), we have an exact sequence of vector bundles:
$$
0 \rightarrow \mathcal{E}_{\CC'}(-1) \overset{\alpha_L}\longrightarrow  \mathcal{E}_{\CC} \rightarrow \mathcal{O}_L(1-|\CC' \cap L| -\epsilon(\CC',L)) \rightarrow 0.
$$
In particular, we have a surjective map 
$$ \mathcal{E}_{\CC} \rightarrow \mathcal{O}_L(1-|\CC' \cap L| -\epsilon(\CC',L)).$$
 Tensor this map by $\mathcal{O}_L$, to obtain again a surjection,  
 $$\mathcal{E}_{\CC} \otimes \mathcal{O}_L  \rightarrow \mathcal{O}_L(1-|\CC' \cap L| -\epsilon(\CC',L)).$$
  By \cite[Theorem 2.8]{MP2}, we know the possible splitting types onto an arbitrary line of the vector bundle $ \mathcal{E}_{\CC}$ associated to a plus-one generated curve $\CC$.  More precisely,  $\mathcal{E}_{\CC} \otimes \mathcal{O}_L$ must be one of the three: $ \mathcal{O}_L(-a) \otimes  \mathcal{O}_L(1-b)$, $ \mathcal{O}_L(1-a) \otimes  \mathcal{O}_L(-b)$ or $ \mathcal{O}_L(-l) \otimes  \mathcal{O}_L(l+1-a-b)$.

Then the same argument as in the proof  of \cite[Proposition 2.5]{MV} proves our claim.
\end{proof}
 
\section{Main results}
\label{sec:main}

We start by looking at CL-arrangements, which were our initial object of study for this note, as they naturally extend line arrangements. Another motivation to start with the class of CL-arrangements is that the proofs of our statements concerning CL-arrangements are then easily extended to reduced curves.

Let 
$$ \CC = \bigcup _i C _i \cup \bigcup _j L_j$$
 be the decomposition into irreducible components of a CL-arrangement $\CC$, where $C_i$'s are smooth conics and $L_j$'s lines. 
 
 \noindent We will generally follow the definitions and constructions from \cite{ST}. 
 
 All conics referred to in this section are considered to be  {\it smooth} conics.\\
 
 For an arbitrary smooth conic $C= C_{i_0}$ in the arrangement $\CC$ define {\it the associated triple  
 $$(\CC, \CC', \CC'')$$
  with respect to the conic $C$} by
 
 $$
 \CC'  = \CC \setminus \{C_{i_0}\}  \coloneqq \bigcup_{i \neq i_0} C_i \cup \bigcup_{j} L_j
 $$
 and 
 $$
 \CC'' \coloneqq \CC' \cap C.
 $$
 
 We will sometimes write $L \in \CC$, respectively $C \in \CC$,  meaning the line $L$, respectively the conic $C$, are  irreducible components of the CL-arrangement $\CC$. We will use the same convention if $C, L$ are irreducible components of a curve $\CC$.
 
 
If $\CC, \CC'$ have only  quasihomogeneous singularities, we will call  $(\CC, \CC', \CC'')$ a {\it  quasihomogeneous triple}. 

 Assume without loss of generality that the (smooth) conic $C$ is defined by the homogeneous degree $2$ polynomial $f_C= y^2-xz$.
 Denote  
 $$k \coloneqq  |\CC''|= | \CC' \cap C|.$$
  Denote moreover  by  $\cE_{\CC},  \cE_{\CC'}$ the sheafification of the $S$-modules of logarithmic derivations $D_0(\CC)$, respectively $D_0(\CC')$. They are locally free sheaves, i.e. rank $2$ vector bundles.\\

By \cite[Proposition 3.7]{ST}, we have the short exact sequence:

\begin{equation}
\label{eq:sheaves_exact}
0 \rightarrow \cE_{\CC'}(-2)  \overset{\cdot f_C}\rightarrow \cE_{\CC}  \rightarrow i_* \OO_{\PP^1}(-k) \rightarrow 0
\end{equation}

where $\PP^1$ is parametrized  by $[s:t]$ and $i: \PP^1  \xrightarrow{ [s^2:st:t^2]} \PP^2$, $i$ being the inclusion of the conic $C$ into $\PP^2$.

For a quasihomogeneous triple $(\CC, \CC', \CC'')$, when $\CC$ is free, from \cite[Lemma 3.6, Proposition 3.7 and the proof of Lemma 3.10]{ST} one can extract formulas for the Hilbert series of the logarithmic derivations module $D_0(\CC')$, in terms of $k$ and of the exponents of the free CL-arrangement $\CC$, see Proposition \ref{prop:HS_del}.  The computations in \cite{ST} are based on the existence of the short exact sequence of sheaves \eqref{eq:sheaves_exact}. 

Since \cite[Lemma 3.10]{ST} gives the recipe for computing  the Hilbert series for the dual $\cE^{\vee}_{\CC'}$ of  $\cE_{\CC'}$, one has to take into account the relation between a rank $2$ vector bundle and its dual, which holds for an arbitrary rank $2$ vector bundle $\cE$:
$$
\cE^{\vee}  \simeq \cE(-c_1(\cE))
$$
where $c_1(\cE)$ is the first Chern number of $\cE$.
Finally, recall that, for a reduced curve $\CC$ in $\PP^2$, 
$$c_1(\cE_{\CC}) = 1-\deg(\CC).$$

\begin{prop}
\label{prop:HS_del}
Let  $(\CC, \CC', \CC'')$ be a quasihomogeneous triple with respect to the conic $C \in \CC$ such that $\CC$ is free with exponents $(a,b)$. Let $k = |\CC''|$ and $\deg(\CC) = d$. 

  If $k=2m$, then 
$$HS(D_0(\CC'))(t) = \frac{t^a+t^b+t^{d-3-m} - t^{d-1-m}}{(1-t)^3}.$$
 If $k=2m+1$, then 
$$HS(D_0(\CC'))(t) = \frac{t^a+t^b+2t^{d-3-m} - 2t^{d-2-m}}{(1-t)^3}.$$

\end{prop}

\begin{theorem}
\label{thm:del_conic}
Let  $(\CC, \CC', \CC'')$ be a quasihomogeneous triple with respect to the conic $C \in \CC$ such that $\CC$ is free with exponents $(a,b), \; a\leq b$. Let $k = |\CC''|$. 
\begin{enumerate}
\item  $k=2m$ 
	\begin{enumerate}
	\item $ m \in \{a,b\}$ if and only if $\CC'$ is free. It this case, $\CC'$ has exponents $(a, b-2)$ if $m=a$, respectively $(a-2, b)$ if $m=b$.
	\item  If $ m \notin \{a,b\}$, then:
		\begin{enumerate}
		\item  $m=a-1$ if and only if  $\CC'$ is plus-one generated. In this case, $\CC'$ has exponents $(a,b-1)$ and level $b$.
		\item $m \leq a-2$ if and only if  $\CC'$ is neither free nor plus-one generated.
		\end{enumerate}
	\end{enumerate}
\item $k=2m+1$
\begin{enumerate}
	\item $a=b=m+1$ if and only if $\CC'$ is free. It this case, $\CC'$ has exponents $(m,m)$.
	\item  If $a \neq m+1$ or $b \neq m+1$, then:
		\begin{enumerate}
		\item  $m=a-1$ if and only if  $\CC'$ is plus-one generated. In this case, $\CC'$ has exponents $(a,b-1)$ and level $b-1$.
		\item $m \leq a-2$ if and only if  $\CC'$ is neither free nor plus-one generated.
		\end{enumerate}
	\end{enumerate}
\end{enumerate}
\end{theorem}

\begin{proof}

In the notations from \eqref{eq: res_gen_deriv}, let $d_1, d_2$ be  the minimal degrees of the homogeneous generators from a minimal set of generators of $D_0(\CC')$, i.e. the first two generalized exponents.  \\ 

 {\it Case $k=2m$.}
 
  In this case, we know from Proposition \ref{prop:HS_del} that the Hilbert series of $D_0(\CC')$ is given by the formula 
\begin{equation}
\label{eq:HS'_even}
 HS(D_0(\CC'))(t) = \frac{t^a+t^b+t^{d-3-m} - t^{d-1-m}}{(1-t)^3}.
 \end{equation} 
 $\CC'$ is free only if there is a cancellation in the numerator of the above formula. 
  If a cancellation happens, then either  $a=d-1-m$ or $b=d-1-m$.  Assume first that $a=d-1-m$. 
  Since $d=a+b+1,$
 it follows that $$b = m.$$ Then 
 $$HS(D_0(\CC'))(t) = \frac{t^{a-2}+t^b}{(1-t)^3}$$
  and 
  $$a-2+b = d-3 = \deg(\CC')-1.$$
   By Theorem \ref{thm: min_gens_sum}, 
 $\CC'$ is free with exponents $(a-2, b)$, which just recovers \cite[Theorem 3.4, case $k=2m, \; (2) \Rightarrow (1)$]{ST}, with a slightly shortened argument.
 
 The case  $b=d-1-m$ is analogous and in this case one gets that $\CC'$ is free with exponents $(a, b-2)$.
 
 Conversely, if $m \in \{a,b\}$, then, assuming  for instance $m=a$, 
 $$HS(D_0(\CC'))(t) = \frac{t^{b-2}+t^a}{(1-t)^3},$$
  and, by  Theorem \ref{thm: min_gens_sum}, $(a, b-2)$ are the minimal degrees of the homogeneous generators from a minimal set of generators of $D_0(\CC')$ and $\CC'$ is free with exponents $(a, b-2)$.
  
   The case $m=b$ is completely analogous and one gets $\CC'$ free with exponents $(a-2, b)$.\\
 
 Assume now  $m \notin \{a,b\}$, which is equivalent to $(d-m-1) \notin \{a,b\}$.  Then there is no cancellation in the numerator of the Hilbert series.  This implies that $\CC'$ is not free.
 
Considering the general type \eqref{eq: res_gen_deriv}  of  resolution for $D_0(\CC')$, we get 
\begin{equation}
\label{eq:incld1d2}
\{d_1,d_2\} \subset \{a,b,d-3-m\},
\end{equation}
so, by the minimality of $d_1, d_2$,  $$d_1+d_2 \leq a+b$$
  i.e.,  keeping in mind that $\deg(\CC') = d-2 = a+b-1,$
   $$d_1+d_2 \leq \deg(\CC')+1.$$
 Since $\CC'$ is not free, by Theorem \ref{thm: min_gens_sum},  
 $$d_1+d_2 \geq \deg(\CC'),$$
 Summing up the two inequalities,
$$ \deg(\CC') \leq d_1+d_2 \leq \deg(\CC')+1.$$
  
Consider first the case  $d_1+d_2 = \deg(\CC')$. By Theorem \ref{thm: min_gens_sum} this is equivalent to  $\CC'$ being plus-one generated with exponents $(d_1, d_2)$.
 Since at the same time  $d_1+d_2 = a+b-1$, we cannot have $\{d_1, d_2\} = \{a,b\}$.  So necessarily 
 $$d-3-m < b, \; \{d_1,d_2\} = \{a, d-3-m\}.$$
This implies  $b = d-2-m$.  An inspection of the Hilbert series formula of $D_0(\CC')$ shows that the plus-one generated $\CC'$ is of level $b$. To sum up, $\CC'$ is plus-one generated with exponents $(a,b-1)$ and level $b$. Moreover, notice that $b = d-2-m$ is equivalent to  $m=a-1$.
  
Consider now the case  $d_1+d_2 =\deg(\CC')+1$. This implies that $\CC'$ is neither free nor plus-one generated, by Theorem \ref{thm: min_gens_sum}. Since $d_1+d_2 = a+b $,  by \eqref{eq:incld1d2}  we necessarily have 
$$\{d_1,d_2\} = \{ a, b\}.$$
Moreover, since $a \leq b$, it follows $b \leq  d-3-m$, i.e. $m \leq  a-2$.\\

  Conversely,  we have to consider the cases $m=a-1$ and $m \leq  a-2$.
  
   If $m=a-1$, then $d-3-m = b-1$ and necessarily $\{d_1,d_2\} =  \{a,d-3-m\}$, so $d_1+d_2 = \deg(\CC')$. Again by Theorem \ref{thm: min_gens_sum},  $\CC'$ is plus-one generated. It has exponents $(a, b-1)$ and level $b$, by \eqref{eq:HS'_even}.

If $m \leq  a-2$, i.e.  $b \leq  d-3-m$, then  $\{d_1,d_2\} = \{ a, b\}$, so 
  $$d_1+d_2 =  \deg(\CC')+1.$$ By Theorem \ref{thm: min_gens_sum},  $\CC'$ is neither free nor plus-one generated. \\
  
  {\it Case $k=2m+1$.}
  
  In this case, we know from Proposition \ref{prop:HS_del} that the Hilbert series of $D_0(\CC')$ is given by the formula 
  \begin{equation}
  \label{eq:HS'_odd}
  HS(D_0(\CC'))(t) = \frac{t^a+t^b+2t^{d-3-m} - 2t^{d-2-m}}{(1-t)^3}.
  \end{equation}
  
  As in the previous case,  a necessary condition for $\CC'$ to be free  is the  cancellation of the term  $- 2t^{d-2-m}$  in the numerator of the above formula, i.e. necessarily  $a=b=m+1$. 
Then $$HS(D_0(\CC'))(t) =  \frac{2t^{m}}{(1-t)^3}$$ 
so $d_1 = d_2 = m$, and, since $\deg(\CC') = 2m+1$, by Theorem \ref{thm: min_gens_sum},  $\CC'$ is free with exponents $(m,m)$.

Conversely, assume $a=b=m+1$, then $HS(D_0(\CC'))(t) =  \frac{2t^{m}}{(1-t)^3}$ and the same argument as above shows that $\CC'$ is free with exponents $(m,m)$.

This recovers \cite[Theorem 3.4, case $k=2m+1, \; (1)$]{ST}.\\

Assume now  $a \neq  m+1$ or $b \neq m+1$ By the above argument, this is equivalent to $\CC'$ being not free. Then at least one term $-t^{d-2-m}$ does not cancel in the numerator of the Hilbert series.  
We have 
  \begin{equation}
 \label{eq:multiset}
 \{d_1,d_2\} \subset \{a,b,d-3-m\},
 \end{equation}
so  $$d_1+d_2 \leq a+b.$$

 Since $\CC'$ is not free, by Theorem \ref{thm: min_gens_sum},  $d_1+d_2 \geq \deg(\CC')$, so, since $ \deg(\CC') = a+b-1$, 
$$ \deg(\CC') \leq d_1+d_2 \leq \deg(\CC')+1.$$
 $$ a+b-1 \leq d_1+d_2 \leq a+b.$$

 Notice that $b=m+1$ and $a \neq m+1$ cannot happen simultaneously. If we assume the contrary, we must have $a<b, a = d-2-m$. So $b> d-2-m$, and this is in contradiction to the formula \eqref{eq:HS'_odd},
  which in this case would translate into
 $$ HS(D_0(\CC'))(t) = \frac{t^{b > d-2-m}+2t^{d-3-m} - t^{d-2-m}}{(1-t)^3}.$$
 
 As we have seen,  $d_1+d_2 \in  \{\deg(\CC'), \deg(\CC')+1 \}$.
 The equality  $d_1+d_2 = \deg(\CC')$ is equivalent, by Theorem \ref{thm: min_gens_sum}, to  $\CC'$ being plus-one generated with exponents $(d_1, d_2)$.
 Since at the same time  $d_1+d_2 = a+b-1$, then, by \eqref{eq:multiset}, necessarily $$d-3-m < b,$$ 
  hence either 
  $$\{d_1,d_2\} = \{a, d-3-m\}\;  \text{and}\;  b = d-2-m,$$
   or 
   $$d_1=d_2=d-3-m.$$

If  $\{d_1,d_2\} = \{a, d-3-m\}$, then, by the Hilbert series formula of $D_0(\CC')$ \eqref{eq:HS'_odd},  the level of $\CC'$  is $b-1$. Moreover, notice that $b = d-2-m$ is equivalent to  $m=a-1$.

If  $d_1=d_2=d-3-m$ and, moreover, $a=d-3-m$, we are back in the previous case. Otherwise, $a > d-3-m$, which leads to $d_1+d_2 \leq a+b-2$, contradiction.\\
  
 Let us prove now the "only if" implication of case (2)(b)(i).  Hence, we assume $m=a-1$ (equivalently, $b = d-2-m$). Since $b \neq m+1$, it follows $a<b$. So $a \leq d-3-m$.
   Then necessarily $\{d_1,d_2\} =  \{a,d-3-m\}$, so $d_1+d_2 = \deg(\CC')$. Again by Theorem \ref{thm: min_gens_sum},  $\CC'$ is plus-one generated.  By \eqref{eq:HS'_odd} its exponents are $(a, b-1)$  and level $ b-1$.\\
  
 Finally, consider the case  $d_1+d_2  = \deg(\CC')+1$. This is equivalent to $\CC'$ being neither free nor plus-one generated, by Theorem \ref{thm: min_gens_sum}.  In particular, by the argument in the above paragraph, $m \neq a-1$, since $m = a-1$ implies $\CC'$ plus-one generated. So,  $b \neq d-2-m$. Then, a cancellation in the numerator of the Hilbert series formula \eqref{eq:HS'_odd} would be possible if and only if $a=d-2-m$. But this would imply $d_1=d_2=d-3-m$ and $d_1+d_2 \leq a+b-2$, contradiction. So, there is no cancellation  in the numerator of the Hilbert series formula \eqref{eq:HS'_odd}.
  Since $d_1+d_2 = a+b $, after taking into account the possible cases arising from the inclusion \eqref{eq:multiset}, it follows that  
 $$\{d_1,d_2\} = \{ a, b\}.$$
Moreover, since $a \leq b$, it follows $b \leq  d-3-m$, i.e. $m \leq  a-2$.\\
  
We have left to prove the 'only if' part of case (2)(b)(ii).  Assume $m \leq  a-2$ (equivalently,  $b \leq  d-3-m$). Then, by   \eqref{eq:HS'_odd}, $\{d_1,d_2\} = \{ a, b\}$, so $d_1+d_2 = \deg(\CC')+1$. By Theorem \ref{thm: min_gens_sum},  $\CC'$ is neither free nor plus-one generated. \\
\end{proof}

Let us try to shed some more  light on the situation when the curve $\CC'$, that is the result of deletion of a conic from a free curve, from 
the above theorem,  is neither free nor plus-one generated. The next couple of statements (\ref{cor:not_free_or_POG} and \ref{rem:del_case_3_detailed}), that give a precise characterization of the  curve $\CC'$ in this situation, are implicit in the proof of Theorem \ref{thm:del_conic}.

\begin{corollary}
\label{cor:not_free_or_POG}
Let $\CC$ be a free CL-arrangement with exponents $(a,b), \; a \leq b$, and  $C$ an arbitrary conic in $\CC$, such that $(\CC, \CC', \CC'')$ is the quasihomogeneous triple with respect to the conic $C$. Let $d_1, d_2$ be the generalized exponents of $\CC'$ defined in  \eqref{eq: gen_exp}.
 If $\CC'$ is neither free nor plus-one generated, then  $$d_1+d_2 = \deg(\CC') +1.$$
\end{corollary}

\begin{proof}
From the proof of  Theorem \ref{thm:del_conic} we already know that $d_1 = a, \; d_2 = b$. Since $\deg(\CC') = \deg(\CC) -2 = a+b-1$, the conclusion follows.
\end{proof}

In fact, we can extract from the proof of  Theorem \ref{thm:del_conic} an even more precise characterisation of the curve $\CC'$ obtained by deletion, when $\CC'$  is neither free nor plus-one generated. 

\begin{remark}
\label{rem:del_case_3_detailed}
In the hypothesis and notations of Theorem \ref{thm:del_conic}, let us describe in more detail the derivations module $D_0(\CC')$ of the curve $\CC'$ resulting by deletion, from cases (1)(b)(ii) and (2)(b)(ii). That is, we consider the cases in which the result of the deletion of a conic from $\CC$  is neither free nor plus-one generated. \\

\noindent Denote  $d'=\deg(\CC')$, and recall that $a+b=d'+1$.\\

In the case  (1)(b)(ii), $D_0(\CC')$ has a resolution of type 
$$0 \rightarrow  S(-d'-1+m) \rightarrow  S(-d'+1+m) \oplus S(-b) \oplus S(-a) \rightarrow D_0(\CC') \rightarrow 0,$$
where $d'-1-m \geq b$. \\

In the case  (2)(b)(ii), $D_0(\CC')$ has a resolution of type 
$$0 \rightarrow  S(-d'+m)^2 \rightarrow  S(-d'+1+m)^2 \oplus S(-b) \oplus S(-a) \rightarrow D_0(\CC') \rightarrow 0,$$
where $d'-1-m \geq b$.\\

This follows straightforward from the proof of Theorem \ref{thm:del_conic}.	
\end{remark}
  
\begin{corollary}
\label{cor:comb_del}
Let $\CC$ be a free CL-arrangement with exponents $(a,b), \; a \leq b$, and  $C$ an arbitrary conic in $\CC$, such that $(\CC, \CC', \CC'')$ is the quasihomogeneous triple with respect to the conic $C$. Let $k =|\CC''|$.  
\begin{enumerate}
\item  If $k=2m$, then the only possible values for $m$ are $m=b$ or $m \leq a$. 
\item $k=2m+1$, then $m=a-1=b-1$ or $m \leq a-1$. 
\end{enumerate}
\end{corollary}

Given a quasihomogeneous triple $(\CC, \CC', \CC'')$, when $\CC'$ is free, from \cite[Lemma 3,6 Proposition 3.7, Lemma 3.8]{ST}, we get formulas for the Hilbert series of the derivations module $D_0(\CC)$ in terms of $|\CC''|$ and the exponents of the free CL-arrangement $\CC'$. The proofs of the results mentioned from \cite{ST} are based on the existence of the short exact sequence of sheaves \eqref{eq:sheaves_exact}.

\begin{prop}
\label{prop:HS_add}
Let  $(\CC, \CC', \CC'')$ be a quasihomogeneous triple with respect to the conic $C$ in $\CC$ such that $\CC'$ is free with exponents $(a',b')$. Let $k = |\CC''|$. 

\noindent  If $k=2m$, then 
\begin{equation}
\label{eq:HS_even}
HS(D_0(\CC))(t) = \frac{t^{a'+2} + t^{b'+2}+t^{m}-t^{m+2}}{(1-t)^3}.
\end{equation}
  If $k=2m+1$, then 
\begin{equation}
\label{eq:HS_odd}
HS(D_0(\CC))(t) = \frac{t^{a'+2} + t^{b'+2}+2t^{m+1}-2t^{m+2}}{(1-t)^3}.
\end{equation}

\end{prop}

\begin{theorem}
\label{thm:add_conic}
Let  $(\CC, \CC', \CC'')$ be a quasihomogeneous triple with respect to the conic $C$ in $\CC$ such that $\CC'$ is free with exponents $(a',b'), \; a' \leq b'$. Let $k = |\CC''|$. 
\begin{enumerate}
\item  $k=2m$ 
	\begin{enumerate}
	\item $ m \in \{a',b'\}$ if and only if $\CC$ is free. In this case, $\CC$ has exponents $(a', b'+2)$ for $m=a'$, respectively $(a'+2, b')$ for $m=b'$.
	\item  If $ m \notin \{a',b'\}$, then:
		\begin{enumerate}
		\item  $m=b'+1$ if and only if  $\CC$ is plus-one generated. In this case, $\CC$ has exponents $(a'+2,b'+1)$ and level $b'+2$.
		\item $m \geq b'+2$ if and only if  $\CC$ is neither free nor plus-one generated.
		\end{enumerate}
	\end{enumerate}
\item $k=2m+1$
\begin{enumerate}
	\item $a'=b'=m$ if and only if $\CC$ is free. It this case, $\CC$ has exponents $(m+1,m+1)$.
	\item  If $a' \neq m$ or $b' \neq m$, then:
		\begin{enumerate}
		\item  $m=b'$ if and only if  $\CC$ is plus-one generated. In this case, $\CC$ has exponents $(a'+2,b'+1)$ and level $b'+1$.
		\item $m \geq b'+1$ if and only if  $\CC'$ is neither free nor plus-one generated.
		\end{enumerate}
	\end{enumerate}
\end{enumerate}
\end{theorem}

\begin{proof}
{\it Case $k=2m$ }

If we assume $\CC$ is free then there must be a cancellation in the Hilbert series formula of $D_0(\CC)$, which happens only if $a'=m$ or $b'=m$. In this case $\CC$ has exponents $(a', b'+2)$ or $(a'+2, b')$. Conversely, if $a'=m$ or $b'=m$, then by Proposition \ref{prop:HS_add} and Theorem \ref{thm: min_gens_sum}, $\CC$ is free with exponents $(a', b'+2)$ or $(a'+2, b')$.\\

 If $ m \notin \{a',b'\}$, then there is no cancellation the Hilbert series formula of $D_0(\CC)$. Let $d_1, d_2$ be defined as in \eqref{eq: res_gen_deriv}, for the resolution of the module  $D_0(\CC)$. Then
 $$\{d_1, d_2\} \subset \{a'+2, b'+2, m\}.$$
By \eqref{eq:HS_even}, necessarily $m \geq b'$. Since $m \neq b'$, it follows $m > b'$. 

If $m = b'+1$ then $\{d_1, d_2\}  = \{a'+2, b'+1\}$, so $d_1+d_2 = \deg(\CC)$. By Theorem \ref{thm: min_gens_sum}, $\CC$ is plus-one generated and by \eqref{eq:HS_even} the exponents of $\CC$ are $(a'+2, b'+1)$ and the level  is $b'+2$. 

Conversely, if $\CC$ is plus-one generated then  $d_1+d_2 = \deg(\CC)$, by Theorem \ref{thm: min_gens_sum}. Since $\deg(\CC) = a'+b'+3$, this implies $\{d_1, d_2\}  = \{a'+2, m\}$ and $m=b+1$.

If $m \geq b'+2$ then $\{d_1, d_2\}  = \{a'+2, b'+2\}$, so $d_1+d_2 = \deg(\CC)+1$. By Theorem \ref{thm: min_gens_sum}, $\CC$ is neither free, nor plus-one generated.

Conversely, if $\CC$ is neither free nor plus-one generated, then $d_1+d_2 > \deg(\CC)= a'+b'+3$. But $\{d_1, d_2\}$ are the smallest two elements in $\{a'+2, b'+2, m\}$, hence $m \geq b'+2$.\\

{\it Case $k=2m+1$}

If $\CC$ is free, then the term $-2t^{m+2}$ from \eqref{eq:HS_odd} must cancel out, hence necessarily $a'=b' = m$. In this case the exponents of $\CC$ are $(m+1,m+1)$. Conversely, if $a'=b' = m$, then 
$$HS(D_0(\CC))(t) = \frac{2t^{m+1}}{(1-t)^3}$$
and, by Theorem \ref{thm: min_gens_sum}, since $2m+2 = a'+b'+2 =  \deg(\CC)-1$, $\CC$ is free. From the above formula for the Hilbert series of $D_0(\CC)$,  $\exp(\CC)=(m+1,m+1)$. \\

Assume in what follows $a' \neq m$ or $b' \neq m$. In any case, by \eqref{eq:HS_odd}, for $d_1, d_2$ as defined in  \eqref{eq: res_gen_deriv},  for the module $D_0(\CC)$, we have the inclusion
\begin{equation}
\label{eq:incl'}
\{d_1, d_2\} \subset \{a'+2, b'+2, m+1\}.
\end{equation}

First let us prove that the case $m=a'$ and $m \neq b'$ is not possible. Assuming the contrary, we get $m=a'<b'$.  But, by \eqref{eq:HS_odd}, $m \geq b'$, contradiction. 
Hence, in any case $m \neq a'$.  This implies, again by \eqref{eq:HS_odd}, that $a'<m$.\\

Consider now the case $m=b'$. Then, still from \eqref{eq:HS_odd}, $\{d_1, d_2\} = \{a'+2, m+1\}$. But $a'+2+m+1 = a'+b'+3 = \deg(\CC)$, so, by Theorem \ref{thm: min_gens_sum}, $\CC$ is plus-one generated. By \eqref{eq:HS_odd}, $\CC$ has exponents $(a'+2, b'+1)$ and level $b'+1$. 

Conversely, if  $\CC$ is plus-one generated, then $d_1 + d_2= \deg(\CC) = a'+b'+3$, by Theorem \ref{thm: min_gens_sum}. Since $a'+2 \leq m+1$,  $\{d_1, d_2\} = \{a'+2, m+1\}$ and $m+1 = b'+1$, i.e., $b'=m$.\\

If $m \neq b'$ (in which case, we also know that $m \neq a'$), then, from \eqref{eq:HS_odd}, $m \geq b'+1$. From the proof above, this is equivalent to $\CC$ being neither free nor plus-one generated. Incidentally,  from $m \geq b'+1$, in this case we get $\{d_1, d_2\}  = \{a'+2, b'+2\}$.
\end{proof}

We can actually say more about the result of the addition of a conic to a free curve, when this result is neither free nor plus-one generated, see \ref{cor:add_not_free_nPOG} and \ref{rem:add_case_3_detailed}. Both statements are implicitly contained in  the proof of Theorem  \ref{thm:add_conic}.

\begin{corollary}
\label{cor:add_not_free_nPOG}
Let $(\CC, \CC', \CC'')$ be a quasihomogeneous triple with respect to an arbitrary conic $C$ in $\CC$ such that $\CC'$ is free with exponents $(a',b'), a' \leq b'$.   Let $d_1, d_2$ be the generalized exponents of $\CC$ defined in  \eqref{eq: gen_exp}.
If $\CC$ is neither free nor plus-one generated, then  $$d_1+d_2 = \deg(\CC) +1.$$
\end{corollary}

\begin{proof}
We know from the proof of Theorem  \ref{thm:add_conic} that $d_1 = a'+2, \; d_2 =b'+2$. Since $ \deg(\CC)  =  \deg(\CC') +2 = a'+b'+3$, the conclusion follows.
\end{proof}

\begin{remark}
\label{rem:add_case_3_detailed}
In the hypothesis and notations of Theorem \ref{thm:add_conic}, we are able to describe a minimal resolution of  the derivations module $D_0(\CC)$ of the curve $\CC$, resulting by addition of a conic to a free curve, in cases (1)(b)(ii) and (2)(b)(ii). That is, in the cases when the result of the addition is neither free nor plus-one generated.\\

In the case  (1)(b)(ii), $D_0(\CC)$ has a resolution of type 
$$0 \rightarrow  S(-m-2) \rightarrow  S(-m) \oplus S(-b'-2) \oplus S(-a'-2) \rightarrow D_0(\CC) \rightarrow 0,$$
where $m \geq b'+2$ and $(a'+2)+(b'+2)=\deg(\CC)+1$.\\

In the case  (2)(b)(ii), $D_0(\CC')$ has a resolution of type 
$$0 \rightarrow  S(-m-2)^2 \rightarrow  S(-m-1)^2 \oplus S(-b'-2) \oplus S(-a'-2) \rightarrow D_0(\CC) \rightarrow 0, $$
where $m \geq b'+1$ and $(a'+2)+(b'+2)=\deg(\CC)+1$.\\

This is an immediate consequence of the proof of Theorem \ref{thm:add_conic}.	 
\end{remark}

\begin{corollary}
\label{cor:comb_add}
Let $(\CC, \CC', \CC'')$ be a quasihomogeneous triple with respect to an arbitrary conic $C$ in $\CC$ such that $\CC'$ is free with exponents $(a',b'), a' \leq b'$. Let $k =|\CC''|$.
\begin{enumerate}
\item  If $k=2m$, then the only possible values for $m$ are $m=a'$ or $m \geq b'$. 
\item $k=2m+1$, then either $m =a'=b'$ or $m\geq b'$. 
\end{enumerate}
\end{corollary}

Taken together, Corollaries \ref{cor:comb_del}, \ref{cor:comb_add}  and Theorem \ref{thm:curve_line_comb} spell a generalization of Abe's Theorem \ref{thm:Abe_restr_free} to quasihomogeneous CL-arrangements. 
More precisely, under quasihomogeneity assumptions on the CL-arrangements appearing in the statement, we have the following result.

\begin{theorem}
\label{thm:weak_comb}
Let $\CC$ be a free CL-arrangement  with exponents $(a,b), \; a \leq b$.
\begin{enumerate}
\item If $L$ is an arbitrary line in $\CC$ and $\CC' \coloneqq \CC \setminus \{L\}$, then either $|\CC' \cap L| = b+1$ or  $|\CC' \cap L| \leq  a+1$.
\item If $L$ is an arbitrary line such that  $L \notin \CC$, then $|\CC \cap L| = a+1$ or  $|\CC \cap L| \geq  b+1$.
\item If $C$ is an arbitrary conic in $\CC$,  let $k$ be the number of singular points of $\CC$ situated on $C$. Then:
	\begin{enumerate}
	\item  If $k=2m$, then the only possible values for $m$ are $m=b$ or $m \leq a$. 
	\item $k=2m+1$, then $m=a-1=b-1$ or $m \leq a-1$. 
	\end{enumerate}
\item  If $C$ is an arbitrary conic such that  $C \notin \CC$, let $k = |\CC \cap C|$.  Then:
\begin{enumerate}
\item  If $k=2m$, then the only possible values for $m$ are $m=a$ or $m \geq b$. 
\item $k=2m+1$, then either $m =a=b$ or $m\geq b$. 
\end{enumerate}
\end{enumerate}
\end{theorem}

\begin{proof}
Immediately from Corollaries \ref{cor:comb_del} and \ref{cor:comb_add}  and Theorem \ref{thm:curve_line_comb} .
\end{proof}

 As we will see in Subsection \ref{sub:generalization_Non-quasihomogeneity}, the quasihomogenity assumptions can be dropped from the hypothesis of Theorem \ref{thm:weak_comb}, 
see  the generalization in Theorem \ref{thm:cors_gen}.

\begin{remark}
\label{rem:ST_overlap}
Sub-cases $(1)(a), (2)(a)$  of  Theorems \ref{thm:del_conic} and  \ref{thm:add_conic} only recover \cite[Theorem 3.4]{ST}.
\end{remark}

\subsection{From CL-arrangements to reduced curves}

The definition of  quasihomogeneous triple extends easily from CL-arrangements to curves. Let us recall  here, for convenience, this particular instance of \cite[Definition 1.5]{STY}. 

\begin{definition}
\label{def:trp}
 If $\CC'$ is a reduced curve in $\PP^2$ such that $C$ is a smooth conic which is not an irreducible component of $\CC'$, and $\CC \coloneqq \CC' \cup C, \; \CC'' \coloneqq \CC' \cap C$, we call $(\CC, \CC',  \CC'')$ a quasihomogeneous triple if the singularities of $\CC, \CC'$ are quasihomogeneous.
\end{definition}
 Then, by \cite[Theorem 1.6, Remark 1.8]{STY}, there exists an exact sequence:
\begin{equation}
\label{eq:seq_quasi_curves}
0 \rightarrow \cE_{\CC'}(-2)  \overset{\cdot f_C}\rightarrow \cE_{\CC}  \rightarrow i_* \OO_{\PP^1}(-|\CC''|) \rightarrow 0
\end{equation}
where $i$ is an inclusion of $C$ in $\PP^2$.\\

Our main results can be extended verbatim to this generalized version of  quasihomogeneous triples,  as defined in \ref{def:trp}.
Theorems \ref{thm:del_conic}, \ref{thm:add_conic} and their  corollaries  actually hold as stated in this new, more general setting. The proofs are basically the same, but  using instead of \eqref{eq:sheaves_exact}  the exact sequence \eqref{eq:seq_quasi_curves}.

There is one argument from the original proof that may need justification in this new context. At some point, in \cite[Lemmas 3.8, 3.10]{ST} the fact that a derivations module $D$ is isomorphic to the module of sections of the sheafification of $D$ is used. This is true for derivations modules associated to reduced curves in general, not only for derivations modules associated to CL-arrangements, see for instance \cite[Proposition 2.1]{AD}.

\subsubsection{{\bf Non-quasihomogeneous case}}
\label{sub:generalization_Non-quasihomogeneity}

We may drop the quasihomogeneity assumption on the triple $(\CC, \CC',  \CC'')$  as well, by replacing the exact sequence \eqref{eq:seq_quasi_curves} by the more general exact sequence  from \cite[Theorem 2.3]{Dimca}, see \eqref{eq:seq_gen}. We should also notice that, since $C$ is a  smooth conic, its divisors are completely determined by their degrees.
\begin{equation}
\label{eq:seq_gen}
0 \rightarrow \cE_{\CC'}(-2)  \overset{\cdot f_C}\rightarrow \cE_{\CC}  \rightarrow i_* \OO_{\PP^1}(- |\CC''|-\epsilon(\CC', C) )\rightarrow 0
\end{equation}

The only change to the statements of Theorems \ref{thm:del_conic},  \ref{thm:add_conic} and their respective corollaries 
  and subsequent remarks
  is that 
$$k=|\CC''|+\epsilon(\CC', C),$$
as opposite to  $k = |\CC''|$ in the quasihomogeneous case. The proofs go exactly the same, using the short exact sequence \eqref{eq:seq_gen} instead of \eqref{eq:sheaves_exact}. As a consequence, Corollaries   \ref{cor:not_free_or_POG}, \ref{cor:add_not_free_nPOG} and Remarks  \ref{rem:del_case_3_detailed}, \ref{rem:add_case_3_detailed} hold as stated, for non-quasihomogeneous triples of curves as well. \\

In this wider setting, Theorems \ref{thm:del_conic} and  \ref{thm:add_conic} generalize to Theorems \ref{thm:del_conic_gen} and \ref{thm:add_conic_gen}.

\begin{theorem}
\label{thm:del_conic_gen}
Let  $\CC$ be a free reduced curve with exponents $(a,b), \; a\leq b$, $C$ a smooth conic which is an irreducible component of $\CC$ and $\CC' \coloneqq \CC \setminus \{C\}, \; \CC'' \coloneqq  \CC' \cap C$. Let $k = |\CC''|+ \epsilon(\CC', C)$. 
\begin{enumerate}
\item  $k=2m$ 
	\begin{enumerate}
	\item $ m \in \{a,b\}$ if and only if $\CC'$ is free. It this case, $\CC'$ has exponents $(a, b-2)$ if $m=a$, respectively $(a-2, b)$ if $m=b$.
	\item  If $ m \notin \{a,b\}$, then:
		\begin{enumerate}
		\item  $m=a-1$ if and only if  $\CC'$ is plus-one generated. In this case, $\CC'$ has exponents $(a,b-1)$ and level $b$.
		\item $m \leq a-2$ if and only if  $\CC'$ is neither free nor plus-one generated.
		\end{enumerate}
	\end{enumerate}
\item $k=2m+1$
\begin{enumerate}
	\item $a=b=m+1$ if and only if $\CC'$ is free. It this case, $\CC'$ has exponents $(m,m)$.
	\item  If $a \neq m+1$ or $b \neq m+1$, then:
		\begin{enumerate}
		\item  $m=a-1$ if and only if  $\CC'$ is plus-one generated. In this case, $\CC'$ has exponents $(a,b-1)$ and level $b-1$.
		\item $m \leq a-2$ if and only if  $\CC'$ is neither free nor plus-one generated.
		\end{enumerate}
	\end{enumerate}
\end{enumerate}
\end{theorem}

A small point on a notation in the above theorem, and in what follows: $\CC \setminus \{C\}$ denotes the curve that is the union of all the irreducible components of  $\CC$, other than $C$.\\

 This next result sums up the generalized versions of Corollary \ref{cor:not_free_or_POG} and Remark  \ref{rem:del_case_3_detailed}.

\begin{corollary}
\label{cor:-gen_not_free_or_POG}
Let $\CC$ be a free reduced curve with exponents $(a,b), \; a \leq b$.
 If $C$ is an arbitrary smooth conic which is an irreducible component of $\CC$, let $\CC' = \CC \setminus \{C\}$ and $k =|\CC' \cap C|+\epsilon(\CC', C)$. Assume $\CC'$ is neither free nor plus-one generated and denote $d'=\deg(\CC')$. 
\begin{enumerate}
\item  If $k=2m$, then  $D_0(\CC')$ has a minimal resolution of type
 $$0 \rightarrow  S(-d'-1+m) \rightarrow  S(-d'+1+m) \oplus S(-b) \oplus S(-a) \rightarrow D_0(\CC') \rightarrow 0,$$
 where $d'-1-m \geq b$. 
\item  If $k=2m+1$, then  $D_0(\CC')$ has a minimal resolution of type
$$0 \rightarrow  S(-d'+m)^2 \rightarrow  S(-d'+1+m)^2 \oplus S(-b) \oplus S(-a) \rightarrow D_0(\CC') \rightarrow 0,$$
where $d'-1-m \geq b$.
\end{enumerate}
\end{corollary}

\begin{theorem}
\label{thm:add_conic_gen}
Let  $\CC'$ be a free reduced curve with exponents $(a',b'), \; a' \leq b'$, $C$ a smooth conic which is not an irreducible component of $\CC'$ and $\CC \coloneqq \CC' \cup C, \; \CC'' \coloneqq  \CC' \cap C$. Let $k = |\CC''|+ \epsilon(\CC', C)$. 
\begin{enumerate}
\item  $k=2m$ 
	\begin{enumerate}
	\item $ m \in \{a',b'\}$ if and only if $\CC$ is free. It this case, $\CC$ has exponents $(a', b'+2)$ for $m=a'$, respectively $(a'+2, b')$ for $m=b'$.
	\item  If $ m \notin \{a',b'\}$, then:
		\begin{enumerate}
		\item  $m=b'+1$ if and only if  $\CC$ is plus-one generated. In this case, $\CC$ has exponents $(a'+2,b'+1)$ and level $b'+2$.
		\item $m \geq b'+2$ if and only if  $\CC$ is neither free nor plus-one generated.
		\end{enumerate}
	\end{enumerate}
\item $k=2m+1$
\begin{enumerate}
	\item $a'=b'=m$ if and only if $\CC$ is free. It this case, $\CC$ has exponents $(m+1,m+1)$.
	\item  If $a' \neq m$ or $b' \neq m$, then:
		\begin{enumerate}
		\item  $m=b'$ if and only if  $\CC$ is plus-one generated. In this case, $\CC$ has exponents $(a'+2,b'+1)$ and level $b'+1$.
		\item $m \geq b'+1$ if and only if  $\CC'$ is neither free nor plus-one generated.
		\end{enumerate}
	\end{enumerate}
\end{enumerate}
\end{theorem}

This next result encapsulates the generalized versions of Corollary  \ref{cor:add_not_free_nPOG} 
 and Remark  \ref{rem:add_case_3_detailed}.

\begin{corollary}
\label{cor:+gen_not_free_or_POG}
Let  $\CC'$ be a free reduced curve with exponents $(a',b'), \; a' \leq b'$, $C$ a smooth conic which is not an irreducible component of $\CC'$ and $\CC \coloneqq \CC' \cup C$. Let $k = |\CC' \cap C|+ \epsilon(\CC', C)$.  Assume $\CC$ is neither free nor plus-one generated. 
 \begin{enumerate}
\item  If $k=2m$, then $D_0(\CC)$ has a minimal   resolution of type:
$$0 \rightarrow  S(-m-2) \rightarrow  S(-m) \oplus S(-b'-2) \oplus S(-a'-2) \rightarrow D_0(\CC) \rightarrow 0,$$
where $m \geq b'+2$.
\item  If $k=2m+1$, then $D_0(\CC)$ has a minimal   resolution of type:
$$0 \rightarrow  S(-m-2)^2 \rightarrow  S(-m-1)^2 \oplus S(-b'-2) \oplus S(-a'-2) \rightarrow D_0(\CC) \rightarrow 0, $$
where $m \geq b'+1$. 
\end{enumerate}
\end{corollary}

Theorem \ref{thm:weak_comb}(3)-(4)
  generalizes to:

\begin{theorem}
\label{thm:cors_gen}
Let $\CC$ be a free reduced curve with exponents $(a,b), \; a \leq b$.
\begin{enumerate}
\item If $C$ is an  arbitrary smooth conic which is an irreducible component of $\CC$, let $\CC' = \CC \setminus \{C\}$ and $k =|\CC' \cap C|+\epsilon(\CC', C)$.
Then:
\begin{enumerate}
\item  If $k=2m$, then the only possible values for $m$ are $m=b$ or $m \leq a$. 
\item $k=2m+1$, then $m=a-1=b-1$ or $m \leq a-1$. 
\end{enumerate}
\item If $C$ is an arbitrary smooth conic which is not an irreducible component of $\CC$, let $k =|\CC \cap C|+\epsilon(\CC, C)$. Then:
\begin{enumerate}
\item  If $k=2m$, then the only possible values for $m$ are $m=a$ or $m \geq b$. 
\item $k=2m+1$, then either $m =a=b$ or $m\geq b$. 
\end{enumerate}
\end{enumerate}
\end{theorem}

\begin{remark}
Subsequently to the first versions of this note, the notion of {\it type $2$} curve was introduced, in \cite{ADP}. With this new notion, our Theorems \ref{thm:del_conic_gen}, \ref{thm:add_conic_gen} (and considering Corollaries \ref{cor:-gen_not_free_or_POG}, \ref{cor:+gen_not_free_or_POG}, which are the generalized versions of Corollaries \ref{cor:not_free_or_POG}, \ref{cor:add_not_free_nPOG} and Remarks  \ref{rem:del_case_3_detailed}, \ref{rem:add_case_3_detailed}) state that the deletion, respectively, the addition of a smooth conic, applied to a free curve,  produces as a result either a free, plus-one generated or a type $2$ curve. 
With this,  our Theorem \ref{thm:add_conic_gen}(1) and Corollary \ref{cor:+gen_not_free_or_POG}(1) imply \cite[Theorem 1.19]{ADP}.
\end{remark}

\section{Examples}
\label{sec:exp}

We will present in this section a series of examples. Most of the computations are made using Macaulay2. For some computations Singular was used.\\ 

Let us start with some addition examples.

\begin{example}
\label{ex:1}
Consider the free CL-arrangement $\CC': f=0$, where 
$$ f = (x-y)(x+y)(x^2+y^2-z^2)(2y^2-z^2)(2x^2-z^2).$$ It is free with exponents $(2,5)$. Add to $\CC'$  the smooth conic $$C: x^2+3y^2-2z^2.$$  The triple $(\CC' \cup C, \CC', \CC' \cap C)$ is quasihomogeneous  and $|\CC' \cap C| = 4$. By Theorem \ref{thm:add_conic}, case $(1)(a)$, $\CC$ is free with exponents $(2,7)$.\\

If one adds  the smooth conic $$C: x^2+2y^2-z^2 = 0$$ to  $\CC'$, the resulting CL-arrangement is plus-one generated with exponents $(4,6)$ and level $7$, by Theorem \ref{thm:add_conic},  case $(1)(b)(i)$, since  $k = |\CC' \cap C| = 12$, so $m = 6 = b' +1$. (the triple $(\CC' \cup C, \CC', \CC' \cap C)$ is quasihomogeneous)\\

By the addition of  the smooth conic $$C: x^2+3y^2+7xy-xz-2yz = 0$$  to   $\CC'$, the resulting CL-arrangement is neither free nor plus-one generated, by Theorem \ref{thm:add_conic}, case  $(1)(b)(ii)$, since   the triple $(\CC' \cup C, \CC', \CC' \cap C)$ is quasihomogeneous and $k = |\CC' \cap C| = 16$, so $m = 8 \geq b'+2 = 7$.
\end{example}

Let us now present some deletion examples. 

\begin{example}
Let $\CC: f=0$ be the free arrangement with exponents $(2,5)$ from Example \ref{ex:1}
$$ f = (x-y)(x+y)(x^2+y^2-z^2)(2y^2-z^2)(2x^2-z^2).$$  
 Let $C: x^2+y^2-z^2 = 0$. Then $\CC' \coloneqq  \CC \setminus \{C\}$ is free with exponents $(2,3)$. This follows from Theorem \ref{thm:del_conic}, case $(1)(a)$, since $(\CC, \CC',\CC' \cap C)$ is a quasihomogeneous triple and $k = |\CC' \cap C|=4$.
\end{example}

\begin{example}
Let $\CC: f=0$ 
$$  f = (x-y)(x+y)(y+z)(x^2+y^2-z^2)(2y^2-z^2)(2x^2-z^2).$$  
$\CC$ is free with exponents $(3,5)$. Let  $\CC'= \CC \setminus \{C\}$, where $C: x^2+y^2-z^2 = 0$. Notice that  $(\CC, \CC', \CC' \cap C)$ is a quasihomogeneous triple and $|\CC' \cap C| = 5$.  Then $\CC'$ is plus-one generated with exponents $(3,4)$ and level $4$, by Theorem \ref{thm:del_conic}, case $(2)(b)(i)$.
\end{example}

Let us consider for a change some non-quasihomogeneous examples, for which we will apply the deletion theorem in its general form, i.e. Theorem  \ref{thm:del_conic_gen}.

\begin{example}
\label{ex:del_gen_nfree_npog}
Let $\CC: f=0$, where 
$$
f= (x^2 + 2xy + y^2 + xz)(x^2 + xz + yz)(x^2 + xy + z^2)(x+y-z)y(x+z)\cdot 
$$
$$
(2x+y)(x^2 - y^2 + xz + 2yz)(x^2 + 2xy - xz + yz).
$$
$\CC$ is a free CL-arrangement with exponents $(6,7)$. Consider the conic 
$$C : x^2 - y^2 + xz + 2yz = 0,$$
 which is an irreducible component of $\CC$.
 The conic $C$ intersects $\CC'\coloneqq \CC \setminus \{C\}$ in six points. Two of those singularities are not quasihomogeneous: $P_1 = [0:0:1]$, which is an ordinary singularity where $6$ branches of $\CC$ meet and $P_2 = [1:-2:-1]$, which is an ordinary singularity where $7$ branches of $\CC$ meet. A computation of Milnor and Tjurina numbers of the singularities $P_1, P_2$ in $\CC, \CC'$ shows that 
 $$\epsilon(\CC', C) = \epsilon(\CC', C)_{P_1} + \epsilon(\CC', C)_{P_2} = 1+1=2.$$

Since there exist non-quasihomogeneous singularities, in this case one applies  Theorem \ref{thm:del_conic_gen}. 
$$k= |\CC' \cap C|+\epsilon(\CC', C) = 6+2=8 \text{ and } (a,b) = (6,7)$$
so $m=4 \leq a-2$. So, it follows from  Theorem \ref{thm:del_conic_gen}, case $(1)(b)(ii)$,  that $\CC'$ is neither free, nor plus-one generated. 
\end{example}
 
 \begin{example}
\label{ex:del_free}
Let $\CC: f=0$ be the free CL-arrangement with exponents $(6,7)$ from the previous example:
$$
f= (x^2 + 2xy + y^2 + xz)(x^2 + xz + yz)(x^2 + xy + z^2)(x+y-z)y(x+z)\cdot 
$$
$$
(2x+y)(x^2 - y^2 + xz + 2yz)(x^2 + 2xy - xz + yz).
$$
 Consider the conic in $\CC$
$$C : x^2 + 2xy + y^2 + xz.$$
 We will prove that the CL-arrangement $\CC' \coloneqq \CC \setminus \{C\}$ is free with exponents $(4,7)$.
 
 One has  $|\CC' \cap C| =10$, and two of those $10$ singularities are not quasihomogeneous. They are the points $P_1, P_2$  described in Example \ref{ex:del_gen_nfree_npog}. A computation of the Milnor and Tjurina numbers of the singularities $P_1, P_2$ in $\CC, \CC'$ shows 
 $$\epsilon(\CC', C) = \epsilon(\CC', C)_{P_1} + \epsilon(\CC', C)_{P_2} = 2+2=4.$$

Since $$|\CC' \cap C|+\epsilon(\CC', C) = 10+4=14 \text{ and } (a,b) = (6,7),$$
it follows $m=7 = b$. Apply Theorem \ref{thm:del_conic_gen}, case $(1)(a)$, to conclude that $\CC'$ is indeed free, with exponents $(4,7)$.
 \end{example}
 
\begin{example}
\label{ex:del_npog}
Take again  $\CC: f=0$ to be the free CL-arrangement with exponents $(6,7)$ from the Example \ref{ex:del_gen_nfree_npog}.
$$
f= (x^2 + 2xy + y^2 + xz)(x^2 + xz + yz)(x^2 + xy + z^2)(x+y-z)y(x+z)\cdot 
$$
$$
(2x+y)(x^2 - y^2 + xz + 2yz)(x^2 + 2xy - xz + yz).
$$
 Consider the conic in $\CC$
$$C : x^2 + xz + yz.$$
Let us prove that the CL-arrangement $\CC' \coloneqq \CC \setminus \{C\}$ is plus-one generated. 

The conic $C$ intersects $\CC'$ into $10$  distinct points. One of these singularities is not quasihomogeneous, namely the point $P_1 = [0:0:1]$. A Milnor and Tjurina numbers computation gives
$$\epsilon(\CC', C) = \epsilon(\CC', C)_{P_1} =1.$$

We  can apply now Theorem \ref{thm:del_conic_gen}, case $(2)(b)(i)$:  since 
$$|\CC' \cap C|+\epsilon(\CC', C) = 10+1=11 \text{ and } (a,b) = (6,7),$$
it follows $m=5 = a-1$. So $\CC'$ is plus-one generated, with exponents $(6,6)$ and level $6$.
 \end{example}

\bigskip

Anca~M\u acinic,
Simion Stoilow Institute of Mathematics of the Romanian Academy, 
P.O. Box 1-764, RO-014700 Bucharest, Romania. \\
\nopagebreak
\textit{E-mail address:} \texttt{anca.macinic@imar.ro}


\begin{thebibliography}{00}
 
\bibitem{A2} T.~Abe,
{\em Roots of characteristic polynomials and intersection 
points of line arrangements},
J. Singul. {\bf 8} (2014),100--117.

\bibitem{A:Inv} T.~Abe,
{\em Divisionally free arrangements of hyperplanes},
Inventiones Mathematicae, Vol. 207 (2017), no. 3, 1377--1378.

\bibitem{A1}  T.~Abe,
{\em Deletion theorem and combinatorics of hyperplane arrangements},
Mathematische Annalen, Vol. 373 (2019), no. 1-2, 581--595.

\bibitem{A:POG} T.~Abe,
{\em Plus-one generated and next to free arrangements of hyperplanes},
Int. Math. Res. Not.  (2021),  Vol. {\bf 2021}, Issue 12,  9233--9261.

\bibitem{ADen} T. ~Abe,  G. ~Denham,
{\em  Deletion-Restriction for Logarithmic Forms on Multiarrangements},
 arXiv:2203.04816 (2022)
 
\bibitem{AD} T.~Abe, A.~Dimca,
{\em Splitting types of bundles of logarithmic vector fields along plane curves}, Internat. J. Math. {\bf 29(8)} (2018),  Art. Id. 1850055.

\bibitem{ADP} T.~Abe, A.~Dimca, P.~Pokora,
{\em A new hierarchy for complex plane curves}, 
arXiv:2410.11479

\bibitem{AIM}  T.~Abe, D.~Ibadula, A.~Macinic,
{\em On some freeness-type properties for line arrangements}, 
 Ann. Sc. Norm. Super. Pisa Cl. Sci. (5) Vol. XXV (2024), 427--447.
 
 \bibitem{BD} E. ~Artal Bartolo, A. ~Dimca,
 {\em On fundamental groups of plane curve complements},
  Ann. Univ. Ferrara 61  (2015), 255 -- 262.
 
 \bibitem{CM} J.I.~Cogolludo-Agust\'in, D. Matei, 
 {\em Cohomology algebra of plane curves, weak combinatorial type, and formality}
 Trans. AMS Vol. 364, No. 11 (2012),  5765--5790.

\bibitem{Dimca} A.~Dimca,
{\em On plus-one generated curves arising from free curves},
Bulletin of Mathematical Sciences, https://doi.org/10.1142/S1664360724500073.

\bibitem{DS0} A.~ Dimca, G.~ Sticlaru,
{\em Plane curves with three syzygies, minimal Tjurina curves curves, and nearly cuspidal curves},
 Geom. Dedicata  {\bf 207} (2020), 29--49.
 
 \bibitem{F} W. ~Fulton, Intersection Theory, Springer-Verlag Berlin Heidelberg 1998, DOI: https://doi.org/10.1007/978-1-4612-1700-8
 
 \bibitem{MP1} A. ~M\u{a}cinic, P.~Pokora,
 {\em On plus-one generated conic-line arrangements with simple singularities}
 arXiv: 2309.15228,  https://arxiv.org/abs/2309.15228, to appear in Rendiconti Lincei – Matematica e Applicazioni
 
 \bibitem{MP2}  A. ~M\u{a}cinic, P.~ Pokora,
{\em Addition–deletion results for plus-one generated curves}, J Algebr. Comb. Vol. 60, 723 -- 734 (2024). https://doi.org/10.1007/s10801-024-01350-x

 \bibitem{MV} A. ~M\u{a}cinic, J. ~Vall\`es, 
  {\em A geometric perspective on plus-one generated arrangements of lines},
  Int. J. Math.Vol. 35, No. 10, 2450034  (2024)
  
  \bibitem{P}  P. ~Pokora, 
  {\em $\mathcal{Q}$-conic arrangements in the complex projective plane},
  Proc. Amer. Math. Soc. 151 (2023), 2873--2880.

\bibitem{Saito} K.~Saito,
{\em Quasihomogene isolierte Singularit\"{a}ten von Hyperﬂ\"{a}chen}, Inv. Math., 14 (1971),
123--142.

 \bibitem{STY} H. ~Schenck, H. ~Terao, M. ~Yoshinaga, 
 {\em Logarithmic vector ﬁelds for curve conﬁgurations in $\mathbb{P}^2$ with quasihomogeneous singularities},  Math. Res. Lett. {\bf 25} (2018),  1977--1992.

\bibitem{ST}  H.~Schenck, \c{S}.~Toh\u{a}neanu, 
{\em Freeness of conic-line arrangements in $\mathbb{P}^2$},
 Comment. Math. Helv. {\bf 84(2)} (2009), 235--258.
 
 \bibitem{Terao} H. ~Terao, 
 {\em Arrangements of hyperplanes and their freeness
  I, II},  J. Fac. Sci. Univ. Tokyo 27 (1980), 293--320.

\end{thebibliography}
\end{document}